\documentclass[11pt]{article}

\usepackage[leqno]{amsmath}
\usepackage{amssymb,amsfonts,xfrac,MnSymbol}
\usepackage{amsthm}
\usepackage{hyperref}
\usepackage[capitalise]{cleveref}
\usepackage[textsize=tiny]{todonotes}
\setuptodonotes{fancyline, color=blue!30}
\usepackage{enumitem}
\usepackage{graphicx}
\usepackage{tikz-cd, tikz}
\usepackage{color}
\usepackage{import}
\usepackage{float}
\usepackage{caption}
\usepackage{subcaption}

\numberwithin{equation}{section}

\newtheorem{theorem}{Theorem}[section]

\newtheorem{lemma}[theorem]{Lemma}
\newtheorem{corollary}[theorem]{Corollary}

\newtheorem*{theorem*}{Theorem}
\newtheorem*{claim*}{Claim}
\newtheorem*{proposition*}{Proposition}
\newtheorem*{lemma*}{Lemma}
\newtheorem*{corollary*}{Corollary}

\theoremstyle{definition}
\newtheorem{definition}[theorem]{Definition}

\newtheorem{remark}[theorem]{Remark}
\newtheorem{example}[theorem]{Example}
\newtheorem{question}[theorem]{Question}

\newtheorem*{definition*}{Definition}
\newtheorem*{observation*}{Observation}
\newtheorem*{remark*}{Remark}
\newtheorem*{example*}{Example}
\newtheorem*{question*}{Question}
\newtheorem*{exercise*}{Exercise}
\newtheorem*{fact*}{Fact}
\newtheorem*{notation*}{Notation}

\newcommand{\bbC}{\mathbb{C}}

\newcommand{\bbN}{\mathbb{N}}

\newcommand{\bbQ}{\mathbb{Q}}

\newcommand{\bbZ}{\mathbb{Z}}

\DeclareMathOperator{\image}{Im}

\DeclareMathOperator{\supp}{Supp}

\DeclareMathOperator{\FP}{FP}

\newcommand{\set}[1]{\left\{ #1 \right\}}

\newcommand*{\claimproofname}{Proof}

\usepackage{amssymb,amsfonts,xfrac,MnSymbol}
\crefname{cond}{condition}{conditions}
\creflabelformat{cond}{#2#1\@#3}
\crefname{obs}{observation}{observations}
\creflabelformat{obs}{#2#1\@#3}

\usepackage{commath}
\usepackage{authblk}
\usepackage{blindtext}
\usepackage{enumitem}
\usepackage{soul}
\usepackage[final]{pdfpages}
\usepackage[style=alphabetic]{biblatex}
\usepackage{graphicx} 
\usepackage{xcolor}
\usepackage{mathrsfs}

\makeatletter
\newtheorem*{rep@theorem}{\rep@title}
\newcommand{\newreptheorem}[2]{%
\newenvironment{rep#1}[1]{%
 \def\rep@title{#2 \ref{##1}}%
 \begin{rep@theorem}}%
 {\end{rep@theorem}}}
\makeatother

\newreptheorem{theorem}{Theorem}

\title{On Finiteness of Homological Isoperimetric Functions on Top Dimensions}
\date{}
\author{Eduardo Mart\'inez-Pedroza and Diana Vizcaíno Torres}

\begin{document}

\maketitle

\begin{abstract}
We address a question from \cite{BKV25} regarding the finiteness of the homological $R$-isoperimetric function. 
Let $R$ be a subfield of the complex numbers $\mathbb{C}$ with the absolute value norm.
We prove that for any group $G$ that admits a finite $(n+1)$-dimensional model for $K(G,1)$, the homological $n$-isoperimetric function  of $G$ over $R$ is either linear or takes infinite values. 
In particular, by results of Gersten and Mineyev, in the class of groups admitting a finite $2$-dimensional classifying space, the homological $1$-dimensional isoperimetric function over $R$ only captures hyperbolicity. This follows as a particular case of a more general result proved in this note. 
\end{abstract}

\section{Introduction}

Let us start by recalling the notation used by Bader, Kropholler, and Vankov in~\cite{BKV25}. Let $R$ be a normed ring and let $|\cdot|$ denote its norm. We are  particularly interested in the case that the rational numbers $\mathbb{Q}$ is a subring of $R$ and $|qr|=|q||r|$ for any $q\in \mathbb{Q}$ and $r\in R$; in this case we say that $R$ is a \emph{normed ring that respects the rationals}. This includes the case that $R$ is any subfield of $\bbC$ with the norm arising as the restriction of the absolute value.
 
All modules  considered in this note are left modules. Given a free $R$-module $F$ with fixed free basis $\Lambda$,  the {\em $\ell^1$-norm} on $F$ is defined as
\begin{equation}\label{eq:freenorm}
    \begin{split}
        \left|\sum_{x\in\Lambda}a_x x\right|=\sum_{x\in\Lambda}|a_x|.
    \end{split}
\end{equation}
Abusing notation, we denote the $\ell_1$-norm on $F$ and the norm on $R$ by $|\cdot|$. An important observation for this note is that if the normed ring $R$ respects the rationals, then 
the $\ell_1$-norm on any free $R$-module $F$ satisfies  
\begin{equation}\label{eq:freenorm}
    \begin{split}
        |qa|=|q||a|.
    \end{split}
\end{equation}
for any $q\in \mathbb{Q}$ and $a\in F$.

Let $G$ be a discrete group and let 
 $RG$ be the group ring.
 Regard $RG$ as a free $R$-module with free basis $G$, and equip it with the $\ell^1$-norm induced by the norm on $R$.  Given a free $RG$-module $M$ with basis $\Lambda$, we endow it with the $\ell^1$-norm from $RG$. This is equivalent to consider $M$ as a free $R$-module with basis $\Lambda \times G$ and endow it with the $\ell^1$-norm from $R$.   
 
 Recall that the group  $G$ is of {\em type $\FP_n(R)$} if there exists a projective resolution $$\dots\to P_n\to \dots\to P_0\to R\to 0$$ where $P_i$ are finitely generated projective $RG$-modules for $i\leq n$. A projective resolution $(C_k,\partial_k)$ of the trivial $RG$-module $R$ is said to be \emph{$n$-admissible} if $C_n,C_{n+1}$ are finitely generated free $RG$ modules equipped with fixed free bases. Note that if  $G$ is of type $FP_{n+1}(R)$ then there exists an $n$-admissible projective resolution of the trivial $RG$-module $R$,  see for example \cite[Prop. VIII.4.3]{brown}.

Any finitely generated $RG$-module  $M$ admits   a collection of canonical norms called \emph{filling norms} which are defined as follows.  Let $\pi \colon F\to M$ be a surjective homomorphism of $RG$-modules, where $F$ is a free $RG$-module with a fixed free basis inducing an $\ell_1$-norm $|\cdot|$. The \emph{filling norm} on $M$ with respect to $\pi$ and $|\cdot|$ is defined as
$$
||m||:= \inf_{x\in F,\pi(x)=m}{|x|}.
$$
Any two  filling norms $\|\cdot\|$ and $\|\cdot\|'$ on $M$ are equivalent, that is, there exists a constant $C>0$ such that for all $m\in M$ we have
$$
C^{-1}||m||\leqslant ||m||'\leqslant C||m||,
$$
for a proof see~\cite[Corollary 2.5]{MA20}. From here on, let $\mathbb{R}$ and $\mathbb{R}_{\geq0}$ denote the sets of real numbers and non-negative real numbers respectively. 

\begin{definition}\cite{BKV25}\label{def:main} [Homological $R$-isoperimetric function]
Given an $n$-admissible projective resolution $(C_k,\partial_k)$ of the trivial $RG$-module $R$ we have that $\partial_{n+1}\colon C_{n+1}\twoheadrightarrow \image(\partial_{n+1})$ and $C_{n+1}$ is free, so we can consider the corresponding filling norm, $||\cdot||$, that the surjection induces on $\image(\partial_{n+1})$. 
We define the \emph{$n$-isoperimetric function of $(C_k,\partial_k)$ (with respect to $|\cdot|$)} to be the function $f_n\colon\mathbb{R}_{\geq 0}\to\mathbb{R}_{\geq0}\cup\set{\infty}$ defined by
\begin{align*}
f_n(l)&=\sup\{\|b\|\mid b\in \image(\partial_{n+1}),|b|\leq l\}\\
&=\sup\{\inf\set{|c| : \partial_{n+1}(c)=b} \mid b\in \image(\partial_{n+1}),|b|\leq l\}.
\end{align*}
\end{definition}

We say that the homological isoperimetric function $f_n$ \emph{takes only finite values} if $f_n(l)<\infty$ for every $l\in \mathbb{R}_{\geq0}$.  
If the homological isoperimetric function associated to an $n$-admissible resolution of $G$ takes only finite values, then homological isoperimetric function associated to any other $n$-admissable resolution of $G$ takes only finite values, see~\cite[Lemma 2.12]{BKV25}. 
However, it is not obvious in which cases the homological isoperimetric function takes only finite values. It is know that $f_n$ takes only finite values in the case that $R = \bbZ$ with the usual norm and  $G$ is $FP_n(\bbZ)$, see~\cite[Theorem 1.3]{FlemMart}. In this note we partially address the following question. 

\begin{question}\cite{BKV25}
    For which norms $|\cdot|$ on $R$ does the $n$-isoperimetric function $f_n$ take only finite values?
\end{question}

Now we introduce some terminology. Recall that if $S$ is a $G$-set, the \emph{permutation $RG$-module} $R[S]$ is the free $R$-module with basis $S$ such that the action of an element \(g\in G\) on a basis element \(s\in S\) is defined by the set action \(g\cdot s\), and this   extends \(R\)-linearly to the entire module. The support $\supp(\mu)$ of an element $\mu\in R[S]$ is the subset of elements of $S$ that have a non-zero coefficient in the linear combination defining $\mu$. In the case that the $G$-stabilizer of each $s\in S$ is finite, we say that $R[S]$ is a \emph{proper permutation $RG$-module}. Note that   $S$ has finitely many $G$-orbits if and only if $R[S]$ is a finitely generated $RG$-module. The $\ell_1$-norm on $R[S]$ induced by the free $R$-basis $S$ is called a \emph{standard} $\ell_1$-norm for $R[S]$ and it's denoted by $|\cdot|$.

A group $G$ is called  \emph{$(n+1)$-dimensional admissible over $R$} if there exists a projective resolution  $(C_k,\partial_k)$ of the trivial $RG$-module $R$ such that $C_n$ and $C_{n+1}$ are finitely generated proper permutation modules endowed with standard $\ell_1$-norms and such that $\partial_{n+1}\colon  C_{n+1} \to C_n$ is injective.  

To illustrate the definition observe that any of the following two assumptions implies that $G$ admits an $(n+1)$-dimensional admissible resolution over any $R$: 
\begin{itemize}
    \item $G$ admits an $(n+1)$-dimensional $K(G,1)$ with finitely many cells.
    \item $G$ admits an $(n+1)$-dimensional  cocompact model for $\underline{E}G$.
\end{itemize}
A function $f\colon \mathbb{R}_{\geq0} \to \mathbb{R}_{\geq0}\cup\{\infty\}$ is \emph{bounded by a linear function} if there is $C\in \mathbb{R}$ such that $f(l)\leq Cl+C$ for every $l\in \mathbb{R}_{\geq0}$. The main result of this note is the following.
\begin{theorem}\label{thm:main}
    Let $R$ be a normed division ring that respects the rationals. Let $G$ be an $(n+1)$-dimensional admissible group over $R$.  Then the homological isoperimetric function $f_n$ with respect to $R$ takes only finite values if and only if $f_n$ is bounded by a linear function.
\end{theorem}

The following example illustrates the idea of the theorem.
\begin{example}\cite[Remark 3.8]{MA20}
Consider the group presentation $G = \langle x,y| [x,y]\rangle$ and let  $X$ be the universal cover of the presentation complex. The chain complex of $X$ with rational coefficients yields a 2-admissible projective resolution
\[  0\to C_2(X,\bbQ) \xrightarrow[]{\partial_2} C_1(X,\bbQ) \xrightarrow[]{\partial_1} C_0(X, \bbQ) \to \bbQ \to 0.\]
Consider the rational 1-cycle $a_n= \frac{1}{4n}[x^n,y^n]$ for $n \in \bbN$. Then for the $\ell_1$-norm we have that $|a_n|= 1$ and for the filling norm $\|a_n\|=\frac{1}{4}n$. It follows that 
\[ \sup\{\|\gamma\|  \colon \gamma \in \image (\partial_2),  | \gamma  |  \leq 1\} = \infty,\]
and hence the  homological isoperimetric function $f_1$ takes infinite values.
\end{example}

It is known by classical results of Gersten and Mineyev \cite{Mi02}, that if for a group $G$, the homological isoperimetric function $f_1$ over  $R = \mathbb{Q}, \mathbb{R},$ or $\mathbb{C}$ with the standard norm $|\cdot|_{\textrm{abs}}$ is linear, then $G$ is hyperbolic. 

\begin{corollary}
Let $G$ be a  group that admits a $2$-dimensional model for $K(G,1)$, and let $R = \mathbb{Q}, \mathbb{R},$ or $\mathbb{C}$ with the standard norm $|\cdot|_{\textrm{abs}}$. Then $f_1$ takes only finite values if and only if $G$ is hyperbolic.
\end{corollary}

The question of whether  the (growth rate) class of homological isoperimetric functions 
 are  
quasi-isometry invariants of the group has been an object of recent study. In the case that $R=\bbZ$, this was proved  to be true for groups of type $F_n$ in the PhD thesis of Fletcher \cite{Fle98}, see also the work by 
Young \cite{You11}. Recent work of Bader, Kropholler, and Vankov \cite{BKV25} shows
that it is a quasi-isometry invariant for groups of type $FH_n(R)$, provided the
functions take only finite values. Recently, Weis proved that in the class of groups of type $FP_n(R)$ whose homological $n$-isoperimetric functions take only finite values,  the (growth rate) of the homological $n$-isoperimetric functions is quasi-isometry invariant~\cite[Theorem 5.4]{W26}. The following statement is an immediate corollary of our main result.

 
\begin{corollary}
Let $R$ be a normed ring that respects the rationals, and let $G$ and $H$ be quasi-isometric groups in the class of $(n+1)$-dimensional admissible groups over $R$. Assume that their homological $n$-isoperimetric functions over $R$,  $f_n^G$ and $f_n^H$, take only finite values. Then $f_n^G$ and $f_n^H$ define  the same (growth rate) class. 
\end{corollary}

In our main result on the homological isoperimetric function $f_n$, the cohomological dimension of $G$ with respect to $R$ is at most $n+1$. So our result only considers the top homological isoperimetric function of the group, and this raises the following question.

\begin{question}
Let $R$ be a normed ring that respects the rationals and $G$ a group of type $FP_{n+1}(R)$. Assuming that the  homological $n$-isoperimetric function $f_n$ over $R$ takes only finite values, is $f_n$ is bounded by a linear function?
\end{question}

Let us remark that in~\cite[Definition 3.2]{MA20}, there is a small modification of the Definition~\ref{def:main}  of homological $R$-isoperimetric function $f_n$ in the case that $R$ is a subring of $R$, called in that article the homological filling function and denoted by $FV_{G,R}^{n+1}$. In this case $FV_{G,R}^{n+1}$ yields an invariant that captures more than higher rank hyperbolicity. The modification of ~\cite[Definition 3.2]{MA20} in the case of $R=\bbQ$ is  
\begin{align*}
f_n(l)=\sup\{\|b\|\mid b\in K,|b|\leq l\}
\end{align*}
where $K$ is an \emph{integral part} of $\image(\partial_{n+1})$,  that is, $K$ is a $\bbZ G$-submodule of the $\bbQ G$-module $\image(\partial_{n+1})$, $K$ is finitely generated as a 
$\bbZ G$-submodule, and $K$ generates  $\image(\partial_{n+1})$ as an $\bbQ G$-module. For example, in the case that $G$ is a free abelian group of rank 2, the function $FV_{G,\bbQ}^2$ is quadratic.

\section{Proof of the theorem}

 Let $(R, |\cdot|)$ be a normed, division ring that respects the rationals.
 
Two norms $\|\cdot\|$ and $\|\cdot\|'$ on an $R$-module $M$ are said to be \emph{equivalent} if there exists a constant $C>0$ such that for all $m\in M$, $$C^{-1}\|m\|\leq \|m\|'\leq C\|m\|.$$ By Corollary 2.5 from \cite{MA20} any two filling norms on a finitely generated $R$-module $M$ are equivalent.
\begin{remark}
Let $R$ be a division ring and $N$ be an $R$-module. Let $M$ be a submodule of $N$, both finitely generated. Consider a filling norm $\|\cdot \|$ on $N$ induced by $\pi\colon F\rightarrow N$. Since $R$ is a division ring, there exists a retraction $r\colon N\to M$. Then $r\circ \pi\colon F\to M$ is surjective and hence defines a filling norm $\|\cdot \|'$ on $M$. By Corollary 2.5 from \cite{MA20}, $\|\cdot \|'$ and the restriction of $\|\cdot \|$ to $M$ are equivalent.
\end{remark}

\begin{lemma}\label{G:fin:FV:lin:bounded}
    If $G$ is a finite group, then the homological isoperimetric function $f_n$ over  $R$ is bounded by a linear function, in particular, $f_n$ takes only finite values. 
\end{lemma} 
\begin{proof}
    Since $G$ is finite, there is a  $K(G,1)$ complex with finite $d$-skeleton for every $d$. Let $X$ be the universal cover of such complex, and  consider the  chain complex over $R$
\begin{equation}\label{fin:group:res}
\ldots \xrightarrow{\partial_3} C_2(X; R) \xrightarrow{\partial_2} C_1(X; R) \xrightarrow{\partial_1} C_0(X;R) \xrightarrow{\varepsilon} R\to 0.
\end{equation}
This is an $n$-admissible projective resolution of the trivial $RG$-module $R$.  
Since $C_{n+1}(X;R)$ is finitely generated free $RG$-module, we have that  $\image(\partial_{n+1})$   is a finitely generated  $RG$-module as well.  Since $G$ is finite,  $\image(\partial_{n+1})$ is a finitely generated $R$-module. As $R$ is a division ring, $C_{n}(X;R)$ retracts to $\image(\partial_{n+1})$. Then the $\ell_1$-norm $|\cdot|$ and the filling norm $\|\cdot \|$ over $\image(\partial_{n+1})$ are equivalent. Hence there is a constant $k\geq 0$ such that $\|z\|\leq k|z|$ for any $z\in \image(\partial_{n+1})$. This implies that $f_n(l)\leq k l$ for every $l$. 
\end{proof}

A $G$-set is \emph{cofinite} if it contains only finitely many $G$-orbits, and is  \emph{proper} if the stabilizer of each element is finite. 

\begin{lemma}\label{seq:neq:lin:bounded}
Suppose \( f:\mathbb{N}\to\mathbb{N} \) is not linearly bounded. Then there exists a strictly increasing sequence \( \{n_k\}_{k\in\mathbb{N}} \subseteq \mathbb{N} \) such that \( f(n_k) >k n_k \) for all \( k\in\mathbb{N} \).
\end{lemma}
\begin{proof}
Recall that \(l\colon\mathbb{N}\to\mathbb{N}\) is \emph{linearly bounded} if there exist integers \(K\) and \(M \) such that \(l(n)\le Kn\) for all \(n\ge M\). Since \(f\) is not linearly bounded, for every \(k,m\in\mathbb{N}\) there exists \(n\ge m\) with \(f(n)>k n\). We build \(\{n_k\}\) by induction. By taking \(m=1\) and \(k=1\) we obtain \(n_1\) with \(f(n_1)>n_1\). Suppose \(n_k\) has been chosen with \(f(n_k)>k n_k\). Using that \(f\) is not linearly bounded with \(m=n_k+1\) and \(k\) replaced by \(k+1\), we find \(n_{k+1}\ge n_k+1\) such that
\( 
f(n_{k+1})>(k+1)n_{k+1}.\)
\end{proof}

\begin{lemma}\label{pairwise:disjoint:cycles}
Let \( G \) be a finitely generated infinite group, and let \( S \) be a proper   \( G \)-set. If \( A, B \subseteq S \) are finite subsets, then there exists \( g \in G \) such that \( A \cap gB = \emptyset \).
\end{lemma}
\begin{proof}
 Suppose by contradiction that for every $g\in G$, $A\cap gB \neq \emptyset$. By the pigeon-hole arguent,  there is an infinite  sequence  $g_1,g_2,g_3,\ldots$ of distinct element of $G$ an  elements $a\in A$ and $b\in B$ such that $g_i.b=a$ for every $i$. In particular, $g_1^{-1}g_i$ belongs to the $G$-stabilizer of $b$. This implies $b$ has infinite $G$-stabilizer and hence   the $G$-action on $S$ is not proper.   
\end{proof}
 
\begin{proof}[Proof of Theorem \ref{thm:main}]
Since $G$ is {$(n+1)$-dimensional admissible over $R$}, there exists a projective resolution  $(C_k,\partial_k)$ of the trivial $RG$-module $R$ such that $C_n$ and $C_{n+1}$ are finitely generated proper permutation modules equipped with standard $\ell_1$-norms and such that $\partial_{n+1}\colon  C_{n+1} \to C_n$ is injective. 

Suppose by contradiction that \( f_n \) is not linearly bounded and fix $\epsilon>0$.  
By Lemma \ref{seq:neq:lin:bounded} there exists a  strictly increasing  sequence $\{n_k\}_{k\in\mathbb{N}}$ of positive integers such that $f_n(n_k)>kn_k$ for all $k\in\mathbb{N}$.
By definition of $f_n$, for each \( n_k \in \mathbb{N} \), there exists \( \alpha_k \in \image (\partial_{n+1}) \) such that
\[
 |\alpha_k |  \le n_k,  \qquad 
\|\alpha_k\|+\epsilon  = f_n(n_k) > k n_k.
\]
By definition of \( \|\cdot\| \), for each \( \alpha_k \) there exists \( \mu_k \in C_{n+1} \) such that
\[
\partial_{n+1}(\mu_k) = \alpha_k
\quad \text{and} \quad
 |\mu_k | + \epsilon > \|\alpha_k\|.
\]
For any positive integer $l$ define
$$ \nu_l := \frac{1}{l} \sum_{k=1}^l \frac{1}{n_k}\, \mu_k, $$
and observe that  
\[
\partial_{n+1}(\nu_l)
= \frac{1}{l} \sum_{k=1}^l \frac{1}{n_k}\, \alpha_k.
\]
Since \( f_n \) is not linearly bounded,   Lemma \ref{G:fin:FV:lin:bounded} implies that $G$ must be infinite. This allows us to apply Lemma~\ref{pairwise:disjoint:cycles}, which ensures that \( \alpha_k \) and \( \mu_k \) can be chosen such that the supports $\supp \alpha_i$ and $\supp \alpha_j$ are disjoint   if $i\neq j$, and similarly   
$\supp \mu_i \cap  \supp \mu_j =\emptyset$ if $i\neq j$. Indeed, this can be shown inductively, suppose that $\alpha_1,\ldots ,\alpha_k$ and $\mu_1,\ldots , \mu_k$ have been chosen so they have pairwise disjoint supports. Let 
\[
A=\supp\left(\sum_{i=1}^{k}\frac{1}{n_i}\alpha_i\right)\cup\supp\left(\sum_{i=1}^{k}\frac{1}{n_i}\mu_i\right)  \text{ and } B=\supp\left(\frac{1}{n_{k+1}}\alpha_{k+1}\right)\cup\supp\left(\frac{1}{n_{k+1}}\mu_{k+1}\right).\] Since $A,B\subseteq S_1\sqcup S_2$ are finite subsets, by Lemma~\ref{pairwise:disjoint:cycles} there exists $g\in G$ such that $A\cap gB=\emptyset$. Replace $\alpha_{k+1},\mu_{k+1}$ by $g\alpha_{k+1},g\mu_{k+1}$ respectively, and observe that their $\ell_1$ norms and filling norms do not change since they are $G$-invariant. Hence we have chosen new $\alpha_n$ and $\mu_n$ with the required property.

Since the \( \mu_k \) have pairwise disjoint supports and 
$|\mu_k|>kn_k-2\epsilon$ we have  
\[
|\nu_l| 
= \frac{1}{l} \sum_{k=1}^l \frac{1}{n_k}  |\mu_k | 
\ge \frac{1}{l} \sum_{k=1}^l \frac{1}{n_k}  (k n_k-2\epsilon) 
= \frac{1}{l} \sum_{k=1}^l k - 2\epsilon
= \frac{l+1}{2}-2\epsilon,
\]
where the first equality holds since the normed ring $R$ respects the rationals.
On the other hand, since  $\partial_{n+1}$   is injective   \[
\|\partial_{n+1}(\nu_l)\|  = |\nu_l|,
\]
and since the supports of all $\alpha_k$ are pairwise disjoint, 
\[
|\partial_{n+1}(\nu_l)| = \frac{1}{l}\sum_{k=1}^l\frac{1}{n_k} |\alpha_k | 
\le \frac{1}{l}\sum_{k=1}^l\frac{1}{n_k}n_k = 1,
\]
where the first equality holds again since the normed ring $R$ respects the rationals.
Summarizing, we have that
\[
|\partial_{n+1}(\nu_l)| \le 1
\quad \text{and} \quad
\|\partial_{n+1}(\nu_l)\| \ge \frac{l+1}{2}-2\epsilon.
\]
for every $l$
and therefore
\[
f_n(1) = \infty,
\]
contradicting the assumption that \( f_n \) is takes only finite values. 
 Hence \( f_n \) must be linearly bounded.
\end{proof}

\AtNextBibliography{\scriptsize}
\printbibliography

@article{BKV25,
  author  = {Bader, Shaked and Kropholler, Robert and Vankov, Vladimir},
  title   = {Subgroups of word hyperbolic groups in dimension 2 over arbitrary rings},
  journal = {Journal of the London Mathematical Society},
  volume  = {112},
  number  = {1},
  pages   = {e70230},
  year    = {2025},
  doi     = {10.1112/jlms.70230},
  url     = {https://doi.org/10.1112/jlms.70230}
}

@article {MA20,
    AUTHOR = {Arora, Shivam and Martínez-Pedroza, Eduardo},
     TITLE = {Subgroups of word hyperbolic groups in rational dimension 2},
   JOURNAL = {Groups Geom. Dyn.},
  FJOURNAL = {Groups, Geometry, and Dynamics},
    VOLUME = {15},
      YEAR = {2021},
    NUMBER = {1},
     PAGES = {83--100},
   MRCLASS = {20F67 (20F65 20J05 57M60 57S30)},
  MRNUMBER = {4235748},
MRREVIEWER = {Wenyuan\ Yang},
       DOI = {10.4171/ggd/592},
       %URL = {https://doi.org/10.4171/ggd/592},
    %ISSN = {1661-7207,1661-7215},
}

@misc{W26,
      title={Quasi-isometry Invariance of discrete Higher Filling functions}, 
      author={Jannis Weis},
      year={2026},
      eprint={2601.15140},
      archivePrefix={arXiv},
      primaryClass={math.GR},
      url={https://arxiv.org/abs/2601.15140}
}

@article {Mi02,
    AUTHOR = {Mineyev, Igor},
     TITLE = {Bounded cohomology characterizes hyperbolic groups},
   JOURNAL = {Q. J. Math.},
  FJOURNAL = {The Quarterly Journal of Mathematics},
    VOLUME = {53},
      YEAR = {2002},
    NUMBER = {1},
     PAGES = {59--73},
      ISSN = {0033-5606},
   MRCLASS = {20J06 (20F67)},
  MRNUMBER = {1887670},
       DOI = {10.1093/qjmath/53.1.59},
       URL = {https://doi-org.qe2a-proxy.mun.ca/10.1093/qjmath/53.1.59},
}

@book {brown,
    AUTHOR = {Brown, Kenneth S.},
     TITLE = {Cohomology of groups},
    SERIES = {Graduate Texts in Mathematics},
    VOLUME = {87},
 PUBLISHER = {Springer-Verlag, New York-Berlin},
      YEAR = {1982},
     PAGES = {x+306},
   MRCLASS = {20-02 (18-01 20F32 20J05 55-01)},
  MRNUMBER = {672956},
MRREVIEWER = {Ross\ Staffeldt},
DOI={https://doi.org/10.1007/978-1-4684-9327-6},
NOTE={2012 reprint.}
%ISBN = {0-387-90688-6},
}

@phdthesis{Fle98,
  author  = {Fletcher, James Louis},
  title   = {Homological group invariants},
  school  = {University of Utah},
  year    = {1998},
  note    = {Available at \url{https://search.proquest.com/docview/304455841}},
  type    = {Ph.{D}. thesis}
}

@article{You11,
  author  = {Young, Robert},
  title   = {Homological and homotopical higher-order filling functions},
  journal = {Groups, Geometry, and Dynamics},
  volume  = {5},
  number  = {3},
  pages   = {683--690},
  year    = {2011},
  publisher = {EMS Press},
  doi     = {10.4171/GGD/140},
  url     = {https://ems.press/journals/ggd/articles/4568}
}

@article {FlemMart,
    AUTHOR = {Fleming, Joshua W. and Martínez-Pedroza, Eduardo},
     TITLE = {Finiteness of homological filling functions},
   JOURNAL = {Involve},
  FJOURNAL = {Involve. A Journal of Mathematics},
    VOLUME = {11},
      YEAR = {2018},
    NUMBER = {4},
     PAGES = {569--583},
   MRCLASS = {20F65 (16P99 20J05 57M07)},
  MRNUMBER = {3778913},
MRREVIEWER = {Philippe\ Malbos},
       DOI = {10.2140/involve.2018.11.569},
       %URL = {https://doi.org/10.2140/involve.2018.11.569},
    %ISSN = {1944-4176,1944-4184},
}
\end{document}